\documentclass[11pt, reqno]{amsart}%
\usepackage{graphics,color}
\usepackage{amsfonts}
\usepackage{amsmath}
\usepackage{amssymb}
\usepackage{graphicx}
\usepackage[bookmarks=false]{hyperref}%
\providecommand{\U}[1]{\protect\rule{.1in}{.1in}}
\newtheorem{thm}{Theorem}[section]

\newtheorem{lemma}[thm]{Lemma}
\newtheorem{cor}[thm]{Corollary}

\newtheorem{prop}[thm]{Proposition}

\newtheorem{defn}[thm]{Definition}
\newtheorem{example}[thm]{Example}
\newtheorem{nota}[thm]{Notation}

\numberwithin{equation}{section}
\begin{document}
\title[Multi-state Canalyzing Functions over Finite Fields]{Multi-state Canalyzing Functions over Finite Fields}
\author[Yuan Li, David Murrugarra, John O. Adeyeye , Reinhard Laubenbacher]{Yuan Li$^{1}$, David Murrugarra$^{2}$, John O. Adeyeye $^{3\ast}$, Reinhard
Laubenbacher$^{4}$}
\address{{\small $^{1}$Department of Mathematics, Winston-Salem State University, NC
27110,USA}\\
{\small email: liyu@wssu.edu }\\
{\small $^{2}$Virginia Bioinformatics Institute, Virginia Tech, Blacksburg, VA
24061}\\
USA \\
{\small email: davidmur@vt.edu}\\
$^{3}$Department of Mathematics, Winston-Salem State University, NC 27110,USA,
{\small email: adeyeyej@wssu.edu}\\
{\small $^{4}$Virginia Bioinformatics Institute, Virginia Tech, Blacksburg, VA
24061,USA }\\
{\small email: reinhard@vbi.vt.edu}}
\thanks{$^{\ast}$ Supported by WSSU Research Initiation Program Award, 2010}
\date{}

\begin{abstract}
In this paper, we extend the definition of Boolean canalyzing functions to the
canalyzing functions over finite field $\mathbb{F}_{q}$, where $q$ is a power
of a prime. We obtain the characterization of all the eight classes of such
functions as well as their cardinality. When $q=2$, we obtain a combinatorial
identity by equating our result to the formula in \cite{Win}. Finally, for a
better understanding to the magnitude, we obtain the asymptotes for all the
eight cardinalities as either $n\to\infty$ or $q\to\infty$.

\end{abstract}
\maketitle

\section{Introduction}

\label{sec-intro}

In 1993, canalyzing Boolean rules were introduced by S. Kauffman \cite{Kau1}
as biologically appropriate rules in Boolean network models of gene regulatory
networks. When comparing the class of canalyzing functions to other classes of
functions with respect to their evolutionary plausibility as emergent control
rules in genetic regulatory system, it is informative to know the number of
canalyzing functions with a given number of input variables \cite{Win}.
However, the Boolean network modeling paradigm is rather restrictive, with its
limit to two possible functional levels, ON and OFF, for genes, proteins, etc.
Many discrete models of biological networks therefore allow variables to take
on multiple states. Commonly used discrete multi-state model types are the
so-called logical models \cite{ThomasD'Ari}, Petri nets \cite{Steggles}, and
agent-based models \cite{Pogson}. It was shown in \cite{Veliz-Cuba} and
\cite{Hinkelmann} that many of these models can be translated into the rich
and general mathematical framework of {\emph{polynomial dynamical systems over
a finite field}} ${\mathbb{F}_{q}}$. (Software to carry out this translation
is available at \url{http://dvd.vbi.vt.edu/cgi-bin/git/adam.pl}).

In this paper, we generalize the concept of Boolean canalyzing rules to the
multi-state case, that is, to functions over any finite fields $\mathbb{F}%
_{q}$, thus generalizing the results in \cite{Win}. We provide formulas for
the cardinalities of all the eight classes canalyzing functions. We also
obtain the asymptotes of these cardinalities as either $n\rightarrow\infty$ or
$q\rightarrow\infty$.

%%%%%%%%%%%%%%%%%%%%%%%%%%%%%%%%%%%%%%%%%
%%%%%%%%%%%%%%%%%%%%%%%%%%%%%%%%%%%%%%%

\section{Preliminaries}

\label{2} In this section we introduce the definition of a \emph{canalyzing
function}. Let $\mathbb{F}=\mathbb{F}_{q}$ be a finite field with $q$
elements, where $q$ is a power of a prime. If $f$ is a $n$ variable function
from $\mathbb{F}^{n}$ to $\mathbb{F}$, it is well known \cite{Lid} that $f$
can be expressed as a polynomial, called the algebraic normal form (ANF):
\[
f(x_{1},x_{2},\ldots,x_{n})=\sum_{k_{1}=0}^{q-1}\sum_{k_{2}=0}^{q-1}\cdots
\sum_{k_{n}=0}^{q-1}a_{k_{1}k_{2}\ldots k_{n}}{x_{1}}^{k_{1}}{x_{2}}^{k_{2}%
}\cdots{x_{n}}^{k_{n}}%
\]
where each coefficient $a_{k_{1}k_{2}\ldots k_{n}}\in\mathbb{F}$ is a
constant. The number $k_{1}+k_{2}+\cdots+k_{n}$ is the multivariate degree of
the term $a_{k_{1}k_{2}\ldots k_{n}}{x_{1}}^{k_{1}}{x_{2}}^{k_{2}}\cdots
{x_{n}}^{k_{n}}$ with nonzero coefficient $a_{k_{1}k_{2}\ldots k_{n}}$. The
greatest degree of all the terms of $f$ is called the algebraic degree,
denoted by $deg(f)$. The greatest degree of each individual variable $x_{i}$
will be denoted by $deg(f)_{i}$. Let $[n]=\{1,2,\ldots,n\}$.

It is shown in \cite{Veliz-Cuba} that it is no restriction of generality to
consider models in which the set of states of the model variables have the
algebraic structure of a finite field. The above fact that \emph{any} function
$\mathbb{F}^{n}\longrightarrow\mathbb{F}$ can be represented as a polynomial
makes the results of this paper valid in the most general setting of models
that are given as dynamical systems generated by iteration of set functions.

We now define a notion of canalyzing function in multi-state setting, which is
a straight-forward generalization of the Boolean case.

\begin{defn}
A function $f(x_{1},x_{2},\ldots,x_{n})$ is canalyzing in the $i$th variable
with canalyzing input value $a\in\mathbb{F}$ and canalyzed output value
$b\in\mathbb{F}$ if $f(x_{1},\ldots,x_{i-1},a,x_{i+1},\ldots,x_{n})=b$, for
any $(x_{1},x_{2},\ldots,x_{n})$.
\end{defn}

In other words, a function is canalyzing is there exits one variable $x_{i}$
such that, if $x_{i}$ receives certain inputs, this by itself determines the
value of the function. For the purpose of the proofs below we will need to use
families of canalyzing functions for which part of the specification is fixed,
such as the variable $x_{i}$ or $a$ or $b$ or some combination. For ease of
notation, we will refer to a canalyzing function just as canalyzing if no
additional information is specified. A function that is canalyzing in variable
$i$ with canalyzing input value $a\in\mathbb{F}$ and canalyzed output value
$b\in\mathbb{F}$ will be referred to as $<i:a:b>$ canalyzing.

We introduce an additional concept.

\begin{defn}
$f(x_{1},x_{2},\ldots, x_{n})$ is essential in variable $x_{i}$ if there exist
$r, s\in\mathbb{F}$ such that $f(x_{1},\ldots,x_{i-1},r,x_{i+1},\ldots
,x_{n})\neq f(x_{1},\ldots,x_{i-1},s,x_{i+1},\ldots,x_{n})$.
\end{defn}

\begin{example}
let $q=5$,$n=3$. $f(x_{1},x_{2},x_{3})=2(x_{1}-3)^{3}(x_{1}-2)x_{2}+1$. Then
this function is essential on $x_{1}$ and $x_{2}$ but not essential on $x_{3}%
$. It has algebraic degree 5 with $deg(f)_{1}=4$ and $deg(f)_{2}=1$. Note that
$f$ is canalyzing in $x_{1}$ with canalyzing input value 3 and canalyzed
output value 1, i.e. f is $<1:3:1>$ canalyzing. Note that $f$ is also
$<1:2:1>$ and $<2:0:1>$ canalyzing. Since $f$ is not essential in $x_{3}$ it
cannot be $<3:a:b>$ canalyzing for any $a,b\in\mathbb{F}_{q}$.
\end{example}

If a function has exactly one essential variable, say $x_{i}$, then its ANF
is
\[
f=a_{q-1}x_{i}^{q-1}+...+a_{1}x_{i}+a_{0}
\]
there exist a $j\geq1$ such that $a_{j}\neq0$. There are $q^{q}-q$ many such
functions since all the constants should be excluded. $\frac{1}{q}%
(q^{q}-q)=q^{q-1}-1$ many of them have fixed canalyzed value $b$ for any
$b\in\mathbb{F}$ since each number is equal. In total, there are
$n(q^{q-1}-1)$ many one essential variable canalyzing function with fixed
canalyzed value b for any b since there are $n$ variables.

There is only one constant function with fixed canalyzed value which is itself.

\begin{nota}
For $i\in{0,1,\dots,n}$ and $a,b\in\mathbb{F}_{q}$ we will use the following
notation,\newline$\mathcal{C}^{i}_{a,b}$: The set of all functions that are
canalyzing in the $i$th variable with canalyzing input value $a$ and canalyzed
output value $b$.\newline$\mathcal{C}^{i}_{*,b}$: The set of all functions
that are canalyzing in the $i$th variable with some canalyzing input value in
$\mathbb{F}_{q}$ and canalyzed output value $b$.\newline$\mathcal{C}^{i}%
_{a,*}$: The set of all functions that are canalyzing in the $i$th variable
with canalyzing input value $a$ and some canalyzed output value in
$\mathbb{F}_{q}$.\newline$\mathcal{C}^{*}_{a,b}$: The set of all functions
that are canalyzing on some variable with canalyzing input value $a$ and
canalyzed output value $b$.\newline$\mathcal{C}^{i}_{*,*}$: The set of all
functions that are canalyzing in the $i$th variable with some canalyzing input
value in $\mathbb{F}_{q}$ and some canalyzed output value in $\mathbb{F}_{q}%
$.\newline$\mathcal{C}^{*}_{a,*}$: The set of all functions that are
canalyzing on some variable with canalyzing input value $a$ and some canalyzed
output value in $\mathbb{F}_{q}$.\newline$\mathcal{C}^{*}_{*,b}$: The set of
all functions that are canalyzing on some variable with some canalyzing input
value in $\mathbb{F}_{q}$ and canalyzed output value $b$.\newline%
$\mathcal{C}^{*}_{*,*}$: The set of all functions that are canalyzing on some
variable with some canalyzing input value in $\mathbb{F}_{q}$ and some
canalyzed output value in $\mathbb{F}_{q}$, i.e., this set consists of all the
canalyzing functions.
\end{nota}

We have the following propositions.

\begin{prop}
\label{prop1} $\mathcal{C}^{i}_{a,b_{1}}\bigcap\mathcal{C}^{i}_{a,b_{2}} =
\emptyset$ whenever $b_{1}\neq b_{2}$.
\end{prop}

\begin{prop}
\label{prop2} $\mathcal{C}^{i_{1}}_{*,b_{1}}\bigcap\mathcal{C}^{i_{2}%
}_{*,b_{2}} = \emptyset$ whenever $b_{1}\neq b_{2}$ and $i_{1}\neq i_{2}$.
\end{prop}

\begin{proof}
Let $f\in\mathcal{C}^{i_{1}}_{*,b_{1}}\bigcap\mathcal{C}^{i_{2}}_{*,b_{2}}$,
then there exist $a_{1}$ such that the value of $f$ should be always $b_{1}$
if we let $x_{i_{1}}=a_{1}$. Similarly, there exist $a_{2}$ such that the
value of $f$ should be $b_{2}$ if we let $x_{i_{2}}=a_{2}$. But $b_{1}\neq
b_{2}$, a contradiction.
\end{proof}

With the above notations, we have

$\mathcal{C}^{*}_{*,*}=\bigcup_{b\in\mathbb{F}}\mathcal{C}^{*}_{*,b}%
=\bigcup_{a\in\mathbb{F}}\mathcal{C}^{*}_{a,*}=\bigcup_{i\in[n]}%
\mathcal{C}^{i}_{*,*}$,

$\mathcal{C}^{*}_{*,b}=\bigcup_{i\in[n]}\mathcal{C}^{i}_{*,b}=\bigcup
_{a\in\mathbb{F}}\mathcal{C}^{*}_{a,b}$,

$\mathcal{C}^{*}_{a,*}=\bigcup_{b\in\mathbb{F}}\mathcal{C}^{*}_{a,b}%
=\bigcup_{i\in[n]}\mathcal{C}^{i}_{a,*}$,

$\mathcal{C}^{i}_{*,*}=\bigcup_{a\in\mathbb{F}}\mathcal{C}^{i}_{a,*}%
=\bigcup_{b\in\mathbb{F}}\mathcal{C}^{i}_{*,b}$,

$\mathcal{C}^{i}_{*,b}=\bigcup_{a\in\mathbb{F}}\mathcal{C}^{i}_{a,b}$,

$\mathcal{C}^{i}_{a,*}=\bigcup_{b\in\mathbb{F}}\mathcal{C}^{i}_{a,b}$,

$\mathcal{C}^{*}_{a,b}=\bigcup_{i\in[n]}\mathcal{C}^{i}_{a,b}$.

For any set $S$, we use $|S|$ to stand for its cardinality. We use
$C(n,k)=\frac{n!}{k!(n-k)!}$ to stand for the binomial coefficients.

Obviously, for the above notations, the cardinality are same for different
values of $i$, $a$ and $b$. In other words, we have $|\mathcal{C}^{i_{1}%
}_{a_{1},b_{1}}|=|\mathcal{C}^{i_{2}}_{a_{2},b_{2}}|$, $|\mathcal{C}^{i}%
_{*,b}|=|\mathcal{C}^{j}_{*,c}|$, $|\mathcal{C}^{*}_{a,b}|=|\mathcal{C}%
^{*}_{c,d}|$ and etc.

\section{Characterization and enumeration of canalyzing functions over
$\mathbb{F}$}

\label{3} Similar to \cite{Jar1} we have

\begin{lemma}
\label{lm3.1} $f(x_{1},x_{2},...x_{n})$ is $<i:a:b>$ canalyzing iff

$f(X)=f(x_{1},x_{2},...,x_{n})=(x_{i}-a)Q(x_{1},x_{2},...x_{n})+b$. where
$deg(Q)_{i}\leq q-2$.
\end{lemma}

\begin{proof}
From the algebraic normal form of $f$, we rewrite it as $f=x_{i}^{q-1}%
g_{q-1}(X_{i})+x_{i}^{q-2}g_{q-2}(X_{i})...+x_{i}g_{1}(X_{i})+g_{0}(X_{i})$,
where $X_{i}=(x_{1},...,x_{i-1},x_{i+1},...,x_{n})$. Using long division we
get $f(X)=f(x_{1},x_{2},...,x_{n})=(x_{i}-a)Q(x_{1},x_{2},...,x_{n})+r(X_{i}%
)$. Since $f(X)$ is $<i:a:b>$ canalyzing, we get $f(X)=f(x_{1},...x_{i-1}%
,a,x_{i+1},...,x_{n})=b$ for any $(x_{1},x_{2},...x_{n})$, i.e.,$r(X_{i})=b$
for any $X_{i}$. So $r(X_{i})$ must be the constant $b$. We finished the
necessity. The sufficiency is obvious.
\end{proof}

The above lemma means $f(X)$ is $<i:a:b>$ iff $(x_{i}-a)|(f(X)-b)$.

Now we get our first formula.

\begin{lemma}
\label{lm3.2} For any $i\in\lbrack n]$ , $a$, $b\in\mathbb{F}$, there are
$q^{q^{n}-q^{n-1}}$ many $<i:a:b>$ canalyzing functions. In other words,
$|\mathcal{C}^{i}_{a,b}|=q^{q^{n}-q^{n-1}}$.
\end{lemma}

\begin{proof}
In Lemma \ref{lm3.1}, $Q$ can be any polynomial with $deg(Q)_{i}\leq q-2$. Its
ANF is

$\sum_{k_{1}=0}^{q-1}...\sum_{k_{i}=0}^{q-2}...\sum_{k_{n}=0}^{q-1}%
a_{k_{1}k_{2}...k_{n}}{x_{1}}^{k_{1}}{x_{2}}^{k_{2}}...{x_{n}}^{k_{n}}$. Since
each coefficient has $q$ many choices and there are $(q-1)q^{n-1}%
=q^{n}-q^{n-1}$ monomials, we get what we want.
\end{proof}

Because $\mathcal{C}^{i}_{a,*}=\bigcup_{b\in\mathbb{F}}\mathcal{C}^{i}_{a,b}$,
by Proposition \ref{prop1}, we get

\begin{cor}
\label{Co1} $|\mathcal{C}^{i}_{a,*}|=q(q^{q^{n}-q^{n-1}})=q^{q^{n}-q^{n-1}+1}$.
\end{cor}

\begin{lemma}
\label{lm3.3} For any $\{a_{1},a_{2},...,a_{k}\}\subset\mathbb{F}$,
$f(X)\in\bigcap_{j=1}^{k}\mathcal{C}^{i}_{a_{j},b}$ iff

$f(X)=f(x_{1},x_{2},...,x_{n})=(\prod_{j=1}^{k}(x_{i}-a_{j}))Q(x_{1}%
,x_{2},...,x_{n})+b$, where $deg(Q)_{i}\leq q-k-1$.
\end{lemma}

\begin{proof}
From Lemma \ref{lm3.1}, we know $(x_{i}-a_{j})|(f(X)-b)$ for $j=1,2,...k$. So
is their product since they are pairwise coprime.
\end{proof}

\begin{lemma}
\label{lm3.4} $|\bigcap_{j=1}^{k}\mathcal{C}^{i}_{a_{j},b}|=q^{q^{n}-kq^{n-1}%
}$ for any $\{a_{1},a_{2},...,a_{k}\}\subset\mathbb{F}$.
\end{lemma}

\begin{proof}
This is similar to the proof of Lemma \ref{lm3.2} by Lemma \ref{lm3.3}.
\end{proof}

Note: If $k=q$, the above number is 1. This is because it means $(x_{i}%
-a_{j})|(f(X)-b)$ for all the $a_{j}$, $j=1,2,...,q$, i.e., $(x_{i}^{q}%
-x_{i})|(f(X)-b)$, where $x_{i}^{q}-x_{i}=\prod_{a\in\mathbb{F}}(x_{i}-a)$. So
$f(X)-b=0$ which means $f(x)=b$.

\begin{thm}
\label{th1} For any $i\in[n]$, $b\in\mathbb{F}$, $|\mathcal{C}^{i}%
_{*,b}|=|\bigcup_{a\in\mathbb{F}}\mathcal{C}^{i}_{a,b}|=q^{q^{n}}-(q^{q^{n-1}%
}-1)^{q}$.
\end{thm}

\begin{proof}
By Inclusion and Exclusion Principle, we have $|\mathcal{C}^{i}_{*,b}%
|=|\bigcup_{a\in\mathbb{F}}\mathcal{C}^{i}_{a,b}|=$

$\sum_{a\in\mathbb{F}}|\mathcal{C}^{i}_{a,b}|-\sum_{\{a_{1},a_{2}%
\}\subset\mathbb{F}}|\mathcal{C}^{i}_{a_{1},b}\bigcap\mathcal{C}^{i}_{a_{2}%
,b}|+......$

$+(-1)^{k-1}\sum_{\{a_{1},a_{2},...,a_{k}\}\subset\mathbb{F}}|\bigcap
_{j=1}^{k}\mathcal{C}^{i}_{a_{j},b}|+...+(-1)^{q-1}=C(q,1)q^{q^{n}-q^{n-1}}-$

$C(q,2)q^{q^{n}-2q^{n-1}}+(-1)^{k}C(q,k)q^{q^{n}-kq^{n-1}}+...+1=\sum
_{k=1}^{q}(-1)^{k-1}C(q,k)q^{q^{n}-kq^{n-1}}=$

$q^{q^{n}}\sum_{k=1}^{q}(-C(q,k)(-q^{-q^{n-1}})^{k})=q^{q^{n}}%
(1-(1-q^{-q^{n-1}})^{q})=q^{q^{n}}-(q^{q^{n-1}}-1)^{q}$.
\end{proof}

Similarly,

\begin{lemma}
\label{lmNEW1} For any $\{i_{1},i_{2},...,i_{k}\}\subset[n]$, $f(X)\in
\bigcap_{j=1}^{k}\mathcal{C}^{i_{j}}_{a,b}$ iff

$f(X)=f(x_{1},x_{2},...,x_{n})=(\prod_{j=1}^{k}(x_{i_{j}}-a))Q(x_{1}%
,x_{2},...,x_{n})+b$, where $deg(Q)_{i_{j}}\leq q-1$, $j=1,2,...,k$.
\end{lemma}

\begin{lemma}
\label{lmNEW2} $|\bigcap_{j=1}^{k}\mathcal{C}^{i_{j}}_{a,b}|=q^{(q-1)^{k}%
q^{n-k}}$ for any $\{i_{1},i_{2},...,i_{k}\}\subset[n]$.
\end{lemma}

\begin{thm}
\label{th2} $|\mathcal{C}^{*}_{a,b}|=\sum_{1\leq k\leq n}(-1)^{k-1}%
C(n,k)q^{(q-1)^{k}q^{n-k}}$.
\end{thm}

\begin{proof}
$|\mathcal{C}^{*}_{a,b}|=|\bigcup_{i\in[n]}\mathcal{C}^{i}_{a,b}|=\sum_{1\leq
i\leq n}|\mathcal{C}^{i}_{a,b}|-\sum_{1\leq i<j\leq n}|\mathcal{C}^{i}%
_{a,b}\bigcap\mathcal{C}^{j}_{a,b}|+...+$

$(-1)^{k-1}\sum_{1\leq i_{1}<i_{1}...<i_{k}\leq n}|\bigcap_{j=1}%
^{k}\mathcal{C}^{i_{j}}_{a,b}|+...+(-1)^{n-1}|\bigcap_{j=1}^{n}\mathcal{C}%
^{j}_{a,b}|=$

$C(n,1)q^{(q-1)q^{n-1}}-C(n,2)q^{(q-1)^{2}q^{n-2}}+...$

$+(-1)^{k-1}C(n,k)q^{(q-1)^{k}q^{n-k}}+...+(-1)^{n-1}q^{(q-1)^{n}}=\sum_{1\leq
k\leq n}(-1)^{k-1}C(n,k)q^{(q-1)^{k}q^{n-k}}$.
\end{proof}

\begin{cor}
\label{co2} $|\mathcal{C}^{*}_{a,*}|=q\sum_{1\leq k\leq n}(-1)^{k-1}%
C(n,k)q^{(q-1)^{k}q^{n-k}}$
\end{cor}

\begin{proof}
$\mathcal{C}^{*}_{a,*}=\bigcup_{b\in\mathbb{F}}\mathcal{C}^{*}_{a,b}$, we need
to show $\mathcal{C}^{*}_{a,b_{1}}\bigcap\mathcal{C}^{*}_{a,b_{2}}=\phi$ if
$b_{1}\neq b_{2}$. Suppose $f\in\mathcal{C}^{*}_{a,b_{1}}\bigcap
\mathcal{C}^{*}_{a,b_{2}}$, then there exist $i_{1}$ and $i_{2}\in[n]$ such
that $f\in\mathcal{C}^{i_{1}}_{a,b_{1}}\bigcap\mathcal{C}^{i_{2}}_{a,b_{2}}$
since $\mathcal{C}^{*}_{a,b}=\bigcup_{i\in[n]}\mathcal{C}^{i}_{a,b}$. If
$i_{1}=i_{2}$, we get a contradiction by Proposition \ref{prop1}. If
$i_{1}\neq i_{2}$, we get a contradiction by Proposition \ref{prop2} since
$\mathcal{C}^{i_{1}}_{a,b_{1}}\subset\mathcal{C}^{i_{1}}_{*,b_{1}}$ and
$\mathcal{C}^{i_{2}}_{a,b_{2}}\subset\mathcal{C}^{i_{2}}_{*,b_{2}}$
\end{proof}

Now, we are going to find the formula for the number of all the canalyzing
functions with given canalyzed value $b$. In other words, the formula of
$|\mathcal{C}^{*}_{*,b}|$.

Let $S_{b}=\{\mathcal{C}^{i}_{a,b}|i\in\lbrack n],a\in\mathbb{F}\}$ for any
$b\in\mathbb{F}$. By Inclusion and Exclusion Principle , we have
$|\mathcal{C}^{*}_{*,b}|=|\bigcup_{i\in\lbrack n]}\mathcal{C}^{i}%
_{*,b}|=|\bigcup_{i\in\lbrack n]}\bigcup_{a\in\mathbb{F}}\mathcal{C}^{i}%
_{a,b}|=\sum_{k=1}^{nq}(-1)^{k-1}N_{k}$, where $N_{k}=\sum_{s\subset
S_{b},|s|=k}|\bigcap_{T\in s}T|$.

In order to evaluate $N_{k}$, we write all the elements in $S_{b}$ as the
following $n\times q$ matrix.

$\mathcal{C}^{1}_{a_{1},b}\quad\mathcal{C}^{1}_{a_{2},b}.........\mathcal{C}%
^{1}_{a_{q},b}$

$\mathcal{C}^{2}_{a_{1},b}\quad\mathcal{C}^{2}_{a_{2},b}.........\mathcal{C}%
^{2}_{a_{q},b}$

............................

$\mathcal{C}^{n}_{a_{1},b}\quad\mathcal{C}^{n}_{a_{2},b}.........\mathcal{C}%
^{n}_{a_{q},b}$

For any $s\subset S_{b}$ with $|s|=k$, we will chose $k$ elements from the
above matrix to form $s$.

Suppose $k_{1}$ of its elements are from the first row (there are $C(q,k_{1})$
many ways to do so). Let these $k_{1}$ elements be $\mathcal{C}^{1}_{a_{11}%
,b},\mathcal{C}^{1}_{a_{12}},...,\mathcal{C}^{1}_{a_{1{k_{1}}},b}$.

Suppose $k_{2}$ of its elements are from the second row (there are
$C(q,k_{2})$ many ways to do so). Let these $k_{2}$ elements be $\mathcal{C}%
^{2}_{a_{21},b},\mathcal{C}^{2}_{a_{22}},...,\mathcal{C}^{2}_{a_{2{k_{2}}},b}$.

..............................

Suppose $k_{n}$ of its elements are from the last row (there are $C(q,k_{n})$
many ways to do so). Let these $k_{n}$ elements be $\mathcal{C}^{n}_{a_{n1}%
,b},\mathcal{C}^{n}_{a_{n2}},...,\mathcal{C}^{n}_{a_{n{k_{n}}},b}$.

$k_{1}+k_{2}+...+k_{n}=k, 0\leq k_{i}\leq q, i=1,2,...,n$.

Similarly to lemma \ref{lm3.3}, we have

\begin{lemma}
\label{lm3.5} $f\in\bigcap_{T\in s}T$ iff $f=Q(\prod_{i=1}^{k_{1}}%
(x_{1}-a_{1i}))(\prod_{i=1}^{k_{2}}(x_{2}-a_{2i}))...(\prod_{i=1}^{k_{n}%
}(x_{n}-a_{ni}))+b$, where $deg(Q)_{i}\leq q-k_{i}-1,i=1,2,...,n.$
\end{lemma}

Similarly to Lemma \ref{lm3.4}, we have

\begin{lemma}
\label{lm3.6} $|\bigcap_{T\in s} T|=q^{(q-k_{1})(q-k_{2})...(q-k_{n})}$.
\end{lemma}

Hence,%

\[
N_{k}=\sum_{\substack{k_{1}+k_{2}+...+k_{n}=k}}C(q,k_{1})C(q,k_{2}%
)...C(q,k_{n})q^{(q-k_{1})(q-k_{2})...(q-k_{n})}.
\]

We get

\begin{thm}
\label{th3} For any $b\in F$, we have
\[
|\mathcal{C}^{*}_{*,b}|=\sum_{\substack{k=1}}^{nq}(-1)^{k-1}(\sum
_{\substack{k_{1}+k_{2}+...+k_{n}=k}}(\prod_{\substack{j=1}}^{n}%
C(q,k_{j}))q^{\prod_{\substack{j=1}}^{n}(q-k_{j})})
\]

\end{thm}

In summary: For the following eight different classes of canalyzing functions,

$\mathcal{C}^{i}_{a,b}$, $\mathcal{C}^{i}_{*,b}$, $\mathcal{C}^{i}_{a,*}$,
$\mathcal{C}^{*}_{a,b}$, $\mathcal{C}^{i}_{*,*}$, $\mathcal{C}^{*}_{a,*}$,
$\mathcal{C}^{*}_{*,b}$, $\mathcal{C}^{*}_{*,*}$ We have found the cardinality
of all except for $\mathcal{C}^{i}_{*,*}$ and $\mathcal{C}^{*}_{*,*}$. We need
more characterizations. We have

\begin{lemma}
\label{lm3.7} Given $\{a_{1},a_{2},...,a_{k}\}\subset\mathbb{F}$ and
$\{b_{1},b_{2},...,b_{k}\}\subset\mathbb{F}$, Then $f\in\bigcap_{j=1}%
^{k}\mathcal{C}^{i}_{a_{j},b_{j}}$ iff $f=Q_{k}(X)\prod_{j=1}^{k}(x_{i}%
-a_{j})+A_{k-1}\prod_{j=1}^{k-1}(x_{i}-a_{j})+...+A_{1}(x_{i}-a_{1})+A_{0}$.
where $A_{0}=b_{1}$, $a_{t}\in F$, $A_{t}$ is determined by $a_{1}%
,...,a_{t+1}$, $b_{1},...,b_{t+1}$, $t=1,2,...,k-1$. $deg(Q_{k})_{i}\leq
q-k-1$.
\end{lemma}

\begin{proof}
For the necessity, we use induction principle.

For $k=1$, it is true by the definition and Lemma \ref{lm3.1}

$f\in\bigcap_{j=1}^{k}\mathcal{C}^{i}_{a_{j},b_{j}}$ implies $f\in
\bigcap_{j=1}^{k-1}\mathcal{C}^{i}_{a_{j},b_{j}}$, by the assumption, we have
$f=Q_{k-1}(X)\prod_{j=1}^{k-1}(x_{i}-a_{j})+A_{k-2}\prod_{j=1}^{k-2}%
(x_{i}-a_{j})+...+A_{1}(x_{i}-a_{1})+A_{0}$.

Since $f\in\mathcal{C}^{i}_{a_{k},b_{k}}$, we get

$f(x_{1},...,x_{i-1},a_{k},x_{i+1},...,x_{n})=Q_{k-1}\prod_{j=1}^{k-1}%
(a_{k}-a_{j})+A_{k-2}\prod_{j=1}^{k-2}(a_{k}-a_{j})+...+A_{1}(a_{k}%
-a_{1})+A_{0}=b_{k}$ for any $X=(x_{1},...,x_{i-1},a_{k},x_{i+1},...,x_{n})$.
In other words,

$Q_{k-1}(x_{1},...,x_{i-1},a_{k},x_{i+1},...,x_{n})=A_{k-1}$ for any
$X=(x_{1},...,x_{i-1},a_{k},x_{i+1},...,x_{n})$.

So $Q_{k-1}\in\mathcal{C}^{i}_{a_{k},A_{k-1}}$. By Lemma \ref{lm3.1}, we have
$Q_{k-1}=(x_{i}-a_{k})Q_{k}+A_{k-1}$,

$f=Q_{k-1}(X)\prod_{j=1}^{k-1}(x_{i}-a_{j})+A_{k-2}\prod_{j=1}^{k-2}%
(x_{i}-a_{j})+...+A_{1}(x_{i}-a_{1})+A_{0}=$

$=((x_{i}-a_{k})Q_{k}+A_{k-1})\prod_{j=1}^{k-1}(x_{i}-a_{j})+A_{k-2}%
\prod_{j=1}^{k-2}(x_{i}-a_{j})+...+A_{1}(x_{i}-a_{1})+A_{0}=$

$=Q_{k}(X)\prod_{j=1}^{k}(x_{i}-a_{j})+A_{k-1}\prod_{j=1}^{k-1}(x_{i}%
-a_{j})+...+A_{1}(x_{i}-a_{1})+A_{0}$.

We finish the proof of necessity.

When $x_{i}=a_{1}$, we have $f=A_{0}=b_{1}$, so $f\in\mathcal{C}^{i}%
_{a_{1},b_{1}}$.

When $x_{i}=a_{2}$, we set $f=A_{1}(a_{2}-a_{1})+A_{0}=b_{2}$, we get a unique
solution for $A_{1}$ such that $f\in\mathcal{C}^{i}_{a_{2},b_{2}}$...........

When $x_{i}=a_{k}$, we get a unique solution for $A_{k-1}$ such that
$f\in\mathcal{C}^{i}_{a_{k},b_{k}}$. In summary, $f\in\bigcap_{j=1}%
^{k}\mathcal{C}^{i}_{a_{j},b_{j}}$.
\end{proof}

From the above Lemma, we immediately obtain

\begin{lemma}
\label{lm3.8} $| \bigcap_{j=1}^{k}\mathcal{C}^{i}_{a_{j},b_{j}}%
|=q^{(q-k)q^{n-1}}$ given $\{a_{1},a_{2},...,a_{k}\}\subset\mathbb{F}$ and
$\{b_{1},b_{2},...,b_{k}\}\subset\mathbb{F}$
\end{lemma}

In order to evaluate $|\mathcal{C}^{i}_{*,*}|$,we need to generalize Lemma
\ref{lm3.7}.

To save space, we just focus on the cardinality in the following lemma.

\begin{lemma}
\label{lm3.9} $a_{11},a_{12},...a_{1k_{1}};a_{21},a_{22},...a_{2k_{2}%
,};.........;a_{r1},a_{r2},...a_{rk_{r}}$ are $k_{1}+k_{2}+...+k_{r}$ distinct
elements of $\mathbb{F}$, $\{b_{1},b_{2},...,b_{k}\}\subset\mathbb{F}$. Then

$|( \bigcap_{j=1}^{k_{1}}\mathcal{C}^{i}_{a_{1j},b_{1}})\bigcap( \bigcap
_{j=1}^{k_{2}}\mathcal{C}^{i}_{a_{2j},b_{2}})......\bigcap( \bigcap
_{j=1}^{k_{r}}\mathcal{C}^{i}_{a_{rj},b_{r}})|=q^{(q-k_{1}-k_{2}%
-...-k_{r})q^{n-1}}$
\end{lemma}

\begin{proof}
By Lemma \ref{lm3.3}, $f\in\bigcap_{j=1}^{k_{1}}\mathcal{C}^{i}_{a_{1j},b_{1}%
}$ iff $f=Q(X)\prod_{j=1}^{k_{1}}(x_{i}-a_{1j})+b_{1}$, $deg(Q)_{i}\leq
q-k_{1}-1$, i.e., we have

$\bigcap_{j=1}^{k_{1}}\mathcal{C}^{i}_{a_{1j},b_{1}}=\{Q(X)\prod_{j=1}^{k_{1}%
}(x_{i}-a_{1j})+b_{1}|\forall Q, deg(Q)_{i}\leq q-k_{1}-1\}$.

Let $f\in\bigcap_{j=1}^{k_{1}}\mathcal{C}^{i}_{a_{1j},b_{1}}$, then
$f=Q(X)\prod_{j=1}^{k_{1}}(x_{i}-a_{1j})+b_{1}$. If we also have
$f\in\mathcal{C}^{i}_{a_{21},b_{2}}$, let $x_{i}=a_{21}$, we get

$f(x_{1}...,x_{i-1},a_{21},x_{i+1},...,x_{n})=Q(x_{1}...,x_{i-1}%
,a_{21},x_{i+1},...,x_{n})\prod_{j=1}^{k_{1}}(a_{21}-a_{1j})+b_{1}=b_{2}$ for
any $(x_{1}...,x_{i-1},a_{21},x_{i+1},...,x_{n})$. The coefficient
$\prod_{j=1}^{k_{1}}(a_{21}-a_{1j})$ is nonzero, so we can solve for $Q$ and
get $Q\in\mathcal{C}^{i}_{a_{21},A_{1}}$ for some $A_{1} \in\mathbb{F}$. Hence
we can write

$Q=(x_{i}-a_{21})Q_{1}+A_{1}$, i.e., $f=(x_{i}-a_{21})Q_{1}\prod_{j=1}^{k_{1}%
}(x_{i}-a_{1j})+O(x_{i})$, where $O(x_{i})$ is a one variable polynomial whose
coefficients are completely determined by $a_{ij},b_{i}$, $deg(Q_{1}%
)_{i}=deg(Q)_{i}-1$, $deg(O(x_{i}))\leq k_{1}$.

Obviously, we can repeat the above process, to get

$f\in( \bigcap_{j=1}^{k_{1}}\mathcal{C}^{i}_{a_{1j},b_{1}})\bigcap(
\bigcap_{j=1}^{k_{2}}\mathcal{C}^{i}_{a_{2j},b_{2}})......\bigcap(
\bigcap_{j=1}^{k_{r}}\mathcal{C}^{i}_{a_{rj},b_{r}})$ iff

$f=\bar{Q}(\prod_{j=1}^{k_{1}}(x_{i}-a_{1j}))(\prod_{j=1}^{k_{2}}(x_{i}%
-a_{2j}))...(\prod_{j=1}^{k_{r}}(x_{i}-a_{rj}))+\bar{O}(x_{i})$. Where
$deg(\bar{Q})_{i}\leq q-k_{1}-k_{2}-...-k_{r}-1$ and $\bar{O}(x_{i})$ is a
uniquely determined polynomial of $x_{i}$ and $deg(\bar{O}(x_{i}))\leq
k_{1}+...+k_{r}-1$. Hence, we know the cardinality is $q^{(q-k_{1}%
-k_{2}-...-k_{r})q^{n-1}}$.
\end{proof}

Now, we are ready to find the cardinality of $\mathcal{C}^{i}_{*,*}$. We have

\begin{thm}
\label{th4}
\[
|\mathcal{C}^{i}_{*,*}|=q!\sum_{k=1}^{q}\frac{(-1)^{k-1}q^{(q-k)q^{n-1}}%
}{(q-k)!}\sum_{k_{1}+...k_{q}=k,0\leq k_{i}\leq q}\frac{1}{k_{1}%
!k_{2}!...k_{q}!}%
\]

\end{thm}

\begin{proof}
$\mathcal{C}^{i}_{*,*}=\bigcup_{b\in\mathbb{F}}\mathcal{C}^{i}_{*,b}%
=\bigcup_{b\in\mathbb{F}}\bigcap_{a\in\mathbb{F}}\mathcal{C}^{i}_{a,b}$.

Let $S_{i}=\{\mathcal{C}^{i}_{a,b}|a,b \in\mathbb{F}\}$, we get $|\mathcal{C}%
^{i}_{*,*}|=\sum_{k=1}^{q^{2}}(-1)^{k-1}N_{k}$. Where $N_{k}=\sum_{s\subset
S_{i},|s|=k}|\bigcap_{T\in s}T|$.

In order to evaluate $N_{k}$, we write all the elements in $S_{i}$ as the
following $q\times q$ matrix.

$\mathcal{C}^{i}_{a_{1},b_{1}} \mathcal{C}^{i}_{a_{2},b_{1}}...\mathcal{C}%
^{i}_{a_{q},b_{1}}$

$\mathcal{C}^{i}_{a_{1},b_{2}} \mathcal{C}^{i}_{a_{2},b_{2}}...\mathcal{C}%
^{i}_{a_{q},b_{2}}$

...........................

$\mathcal{C}^{i}_{a_{1},b_{q}} \mathcal{C}^{i}_{a_{2},b_{q}}...\mathcal{C}%
^{i}_{a_{q},b_{q}}$

For any $s\subset S_{i}$ with $|s|=k$, we will chose $k$ elements from the
above matrix to form $s$.

Suppose $k_{1}$ of it elements are from the first row (there are $C(q,k_{1})$
many ways to do so). Let these $k_{1}$ elements be $\mathcal{C}_{a_{11},b_{1}%
}^{i},\mathcal{C}_{a_{12},b_{1}}^{i},...,\mathcal{C}_{a_{1{k_{1}}},b_{1}}^{i}$.

Suppose $k_{2}$ of its elements are from the second row, we must chose these
elements from different columns, otherwise the intersection will be $\phi$ by
Proposition \ref{prop1} (there are $C(q-k_{1},k_{2})$ many ways to do so). Let
these $k_{2}$ elements be $\mathcal{C}_{a_{21},b_{2}}^{i},\mathcal{C}%
_{a_{22},b_{2}}^{i},...,\mathcal{C}_{a_{2{k_{2}}},b_{2}}^{i}$

.............................

Suppose $k_{q}$ of its elements are from the last row (there are
$C(q-k_{1}-k_{2}-...-k_{q-1},k_{q})$ many ways to do so). Let these $k_{q}$
elements be $\mathcal{C}_{a_{q1},b_{q}}^{i},\mathcal{C}_{a_{q2},b_{q}}%
^{i},...,\mathcal{C}_{a_{q{k_{q}}},b_{q}}^{i}$.

$k_{1}+k_{2}+...+k_{q}=k$, $0\leq k_{i}\leq q$, $i=1,2,...,q$.

$N_{k}=\sum_{s\subset S_{i},|s|=k}|\bigcap_{T\in s}T|$

$=\sum_{k_{1}+...+k_{q}=k,0\leq k_{i}\leq q}C(q,k_{1})C(q-k_{1},k_{2}%
)...C(q-k_{1}-...-k_{q-1},k_{q})I_{k_{1}k_{2}...k_{q}}$, where

$I_{k_{1}k_{2}...k_{q}}=|( \bigcap_{j=1}^{k_{1}}\mathcal{C}^{i}_{a_{1j},b_{1}%
})\bigcap( \bigcap_{j=1}^{k_{2}}\mathcal{C}^{i}_{a_{2j},b_{2}})......\bigcap(
\bigcap_{j=1}^{k_{q}}\mathcal{C}^{i}_{a_{qj},b_{q}})|$.

By Lemma \ref{lm3.9}, we know $I_{k_{1}k_{2}...k_{q}}=q^{(q-k_{1}%
-k_{2}-...-k_{q})q^{n-1}}=q^{(q-k)q^{n-1}}$, this number is zero if $k>q$.

A straightforward computation shows that

$C(q,k_{1})C(q-k_{1},k_{2})...C(q-k_{1}-...-k_{q-1},k_{q})=\frac{q!}%
{k_{1}!k_{2}!...k_{q}!(q-k)!}$.

Hence, we get $|\mathcal{C}^{i}_{*,*}|=\sum_{k=1}^{q^{2}}(-1)^{k-1}N_{k}%
=\sum_{k=1}^{q}(-1)^{k-1}N_{k}=$%

\[
=\sum_{k=1}^{q}(-1)^{k-1}\sum_{k_{1}+...k_{q}=k,0\leq k_{i}\leq q}\frac
{q!}{k_{1}!k_{2}!...k_{q}!(q-k)!}q^{(q-k)q^{n-1}}=
\]

\[
q!\sum_{k=1}^{q}\frac{(-1)^{k-1}q^{(q-k)q^{n-1}}}{(q-k)!}\sum_{k_{1}%
+...k_{q}=k,0\leq k_{i}\leq q}\frac{1}{k_{1}!k_{2}!...k_{q}!}%
\]

\end{proof}

Now we begin to evaluate $|\mathcal{C}^{*}_{*,*}|$. We have

\begin{thm}
\label{th5}
\[
|\mathcal{C}^{*}_{*,*}|=\sum_{k=1}^{q}(-1)^{k-1}U_{k}+\sum_{k=1}%
^{nq}(-1)^{k-1}V_{k}.
\]
where
\[
U_{k}=n\sum_{t_{1}+t_{2}+...+t_{q}=k,0\leq t_{i}\leq q}\frac{q!}{t_{1}%
!t_{2}!...t_{q}!(q-k)!}q^{(q-k)q^{n-1}}=
\]%
\[
\frac{nq!}{(q-k)!}q^{(q-k)q^{n-1}}\sum_{t_{1}+t_{2}+...+t_{q}=k,0\leq
t_{i}\leq q}\frac{1}{t_{1}!t_{2}!...t_{q}!}%
\]

\[
V_{k}=q\sum_{k_{1}+...+k_{n}=k, 0\leq k_{i}\leq k-1,0\leq k_{i}\leq q}%
(\prod^{n}_{j=1}C(q,k_{j}))q^{\prod^{n}_{j=1}(q-k_{j})}%
\]

\end{thm}

\begin{proof}
First $\mathcal{C}^{*}_{*,*}=\bigcup_{i\in[n]}\bigcup_{a\in\mathbb{F}}%
\bigcup_{b\in\mathbb{F}}\mathcal{C}^{i}_{a,b}$, Let $S=\{\mathcal{C}^{i}%
_{a,b}|i\in[n], a,b \in F\}$,

then $|\mathcal{C}^{*}_{*,*}|=\sum_{k=1}^{nq^{2}}(-1)^{k-1}N_{k}$, where
$N_{k}=\sum_{s\subset S,|s|=k}|\bigcap_{T\in s}T|$.

We write all the $nq^{2}$ elements of $S$ as the following $n$ many $q\times
q$ matrices.

$\mathcal{C}^{1}_{a_{1},b_{1}}\mathcal{C}^{1}_{a_{1},b_{2}}...\mathcal{C}%
^{1}_{a_{1},b_{q}}$

$\mathcal{C}^{1}_{a_{2},b_{1}}\mathcal{C}^{1}_{a_{2},b_{2}}...\mathcal{C}%
^{1}_{a_{2},b_{q}}$

..........................

$\mathcal{C}^{1}_{a_{q},b_{1}}\mathcal{C}^{1}_{a_{q},b_{2}}...\mathcal{C}%
^{1}_{a_{q},b_{q}}$

We call this matrix $M_{1}$.

$\mathcal{C}^{2}_{a_{1},b_{1}}\mathcal{C}^{2}_{a_{1},b_{2}}...\mathcal{C}%
^{2}_{a_{1},b_{q}}$

$\mathcal{C}^{2}_{a_{2},b_{1}}\mathcal{C}^{2}_{a_{2},b_{2}}...\mathcal{C}%
^{2}_{a_{2},b_{q}}$

..........................

$\mathcal{C}^{2}_{a_{q},b_{1}}\mathcal{C}^{2}_{a_{q},b_{2}}...\mathcal{C}%
^{2}_{a_{q},b_{q}}$

We call this matrix $M_{2}$.

.........................

.........................

.........................

$\mathcal{C}^{n}_{a_{1},b_{1}}\mathcal{C}^{n}_{a_{1},b_{2}}...\mathcal{C}%
^{n}_{a_{1},b_{q}}$

$\mathcal{C}^{n}_{a_{2},b_{1}}\mathcal{C}^{n}_{a_{2},b_{2}}...\mathcal{C}%
^{n}_{a_{2},b_{q}}$

..........................

$\mathcal{C}^{n}_{a_{q},b_{1}}\mathcal{C}^{n}_{a_{q},b_{2}}...\mathcal{C}%
^{n}_{a_{q},b_{q}}$

We call this matrix $M_{n}$.

We combine all the above $M_{i},i=1,2,...,n.$ to form a $nq\times q$ matrix
$M$ whose first $q$ rows are $M_{1}$, the second $q$ rows are $M_{2}$,..., the
last $q$ rows are $M_{n}$.

We are going to chose $k$ elements from $M$ to form the intersection. In order
to get a possible non empty intersection, we know all these $k$ elements must
come from either the same $M_{i}$(for some fixed $i$) or all of them from the
same column of $M$ by Proposition \ref{prop2}. Inside the fixed $M_{i}$, each
elements must come from different rows by Proposition \ref{prop1}.

Hence, a typical intersection either looks like the one in Lemma \ref{lm3.9}
or the one in Lemma \ref{lm3.5}. But these two cases are not disjoint.

Suppose we chose $k_{i}$ elements from $M_{i}$, $i=1,2,...,n$, $k_{1}%
+k_{2}+...+k_{n}=k$, $0\leq k_{i}\leq k$, $i=1,2,...,n$.

If there exist $i$ such that $k_{i}=k$, then $k_{j}=0,\forall j\neq i$. This
implies the intersection looks like the one in Lemma \ref{lm3.9} and $k\leq q$.

If $0\leq k_{i}\leq k-1, \forall i\in[n]$, then the intersection looks like
the one in Lemma \ref{lm3.5} and $k\leq nq$.

The above two cases are disjoint now. By Lemma \ref{lm3.9} and Lemma
\ref{lm3.6}, we get
\[
N_{k}=\sum_{s\subset S,|s|=k}|\bigcap_{T\in s}T|=\sum_{k_{1}+...+k_{n}=k,0\leq
k_{i}\leq k}=\sum_{\exists i,k_{i}=k}+\sum_{k_{i}\leq k-1,i=1,...n}%
=U_{k}+V_{k}%
\]
where
\[
U_{k}=n\sum_{t_{1}+t_{2}+...+t_{q}=k,0\leq t_{i}\leq q}\frac{q!}{t_{1}%
!t_{2}!...t_{q}!(q-k)!}q^{(q-k)q^{n-1}}=
\]%
\[
\frac{nq!}{(q-k)!}q^{(q-k)q^{n-1}}\sum_{t_{1}+t_{2}+...+t_{q}=k,0\leq
t_{i}\leq q}\frac{1}{t_{1}!t_{2}!...t_{q}!}%
\]

\[
V_{k}=q\sum_{k_{1}+...+k_{n}=k, 0\leq k_{i}\leq k-1,0\leq k_{i}\leq q}%
(\prod^{n}_{j=1}C(q,k_{j}))q^{\prod^{n}_{j=1}(q-k_{j})}%
\]

Hence,
\[
|\mathcal{C}^{*}_{*,*}|=\sum^{nq^{2}}_{k=1}(-1)^{k-1}N_{k}=\sum^{nq^{2}}%
_{k=1}(-1)^{k-1}(U_{k}+V_{k})=\sum^{q}_{k=1}(-1)^{k-1}U_{k}+\sum^{nq}%
_{k=1}(-1)^{k-1}V_{k},
\]

\end{proof}

In the following, we will reduce the formula $|\mathcal{C}^{*}_{*,*}|$ when
$q=2$ and compare it with the one in \cite{Win}.%

\[
|\mathcal{C}^{*}_{*,*}|=\sum_{k=1}^{2}(-1)^{k-1}U_{k}+\sum_{k=1}%
^{2n}(-1)^{k-1}V_{k}.
\]
where%

\[
U_{k}=n\sum_{t_{1}+t_{2}=k, 0\leq t_{i}\leq2}\frac{2!}{t_{1}!t_{2}%
!(2-k)!}2^{(2-k)2^{n-1}}%
\]

\[
V_{k}=2\sum_{k_{1}+...+k_{n}=k, 0\leq k_{i}\leq k-1,0\leq k_{i}\leq2}%
(\prod^{n}_{j=1}C(2,k_{j}))2^{\prod^{n}_{j=1}(2-k_{j})}%
\]
A simple calculation shows that $U_{1}=4n2^{2^{n-1}}=C(n,1)2^{2}2^{2^{n-1}}$
and $U_{2}=4n$.

$V_{1}=0$ since the condition of the sum is not satisfied.%

\[
V_{2}=2\sum_{k_{1}+...+k_{n}=2, 0\leq k_{i}\leq1}(\prod^{n}_{j=1}%
C(2,k_{j}))2^{\prod^{n}_{j=1}(2-k_{j})}=C(n,2)2^{3}2^{2^{n-2}}%
\]
When $3\leq k\leq2n$%

\[
V_{k}=2\sum_{k_{1}+...+k_{n}=k,0\leq k_{i}\leq2}(\prod_{j=1}^{n}%
C(2,k_{j}))2^{\prod_{j=1}^{n}(2-k_{j})}%
\]%
\[
=C(n,k)2^{k+1}2^{2^{n-k}}+\sum_{1\leq t\leq\llcorner\frac{k}{2}\lrcorner
}C(n,t)C(n-t,k-2t)2^{k-2t+1}%
\]
Hence, when $q=2$,%

\[
|\mathcal{C}_{\ast,\ast}^{\ast}|=-4n+\sum_{1\leq k\leq n}(-1)^{k+1}%
C(n,k)2^{k+1}2^{2^{n-k}}%
\]%
\[
+\sum_{3\leq k\leq2n}(-1)^{k-1}\sum_{1\leq t\leq\llcorner\frac{k}{2}\lrcorner
}C(n,t)C(n-t,k-2t)2^{k-2t+1}%
\]
When $n$=1, 2, 3, 4, one can obtain (without calculator) the sequence 4, 14,
120, 3514. These results are consistent with those in \cite{Win}. By
\cite{Win}, the cardinality of $\mathcal{C}_{\ast,\ast}^{\ast}$ should be
\[
|\mathcal{C}_{\ast,\ast}^{\ast}|=2((-1)^{n}-n)+\sum_{1\leq k\leq n}%
(-1)^{k+1}C(n,k)2^{k+1}2^{2^{n-k}}.
\]
So, we obtain the following combinatorial identity (for any positive integer
$n$).%

\[
\sum_{3\leq k\leq2n}(-1)^{k-1}\sum_{1\leq t\leq\llcorner\frac{k}{2} \lrcorner
}C(n,t)C(n-t,k-2t)2^{k-2t+1}=2((-1)^{n}+n)
\]
The left sum should be explained as $0$ if $n=1$. As usual, $C(n,k)$ is 0 if
$k>n$.

For general $q$, from Lemma \ref{lm3.2}, we know $|\mathcal{C}^{i}%
_{a,b}|=q^{q^{n}-q^{n-1}}$, since $\mathcal{C}^{*}_{*,*}=\bigcup_{i\in
[n]}\bigcup_{a\in\mathbb{F}}\bigcup_{b\in\mathbb{F}}\mathcal{C}^{i}_{a,b}$, we
obtain $|\mathcal{C}^{*}_{*,*}|\leq nq^{2}q^{(q-1)q^{n-1}}$.

In order to get an intuitive idea about the magnitude of all the cardinality
numbers, We will find their asymptote as $n\to\infty$ or $q\to\infty$.

We have the following notation

\begin{defn}
$f(x)\overset{x}{\cong} g(x)$ if $\lim_{x\to\infty}\frac{f(x)}{g(x)}=1$.
\end{defn}

Now, we can list all the cardinalities asymptotically, we have

\begin{thm}
\label{th3.19}
\[
|\mathcal{C}^{i}_{a,b}|=q^{(q-1)q^{n-1}};
\]

\[
|\mathcal{C}^{i}_{a,*}|=qq^{(q-1)q^{n-1}};
\]

\[
|\mathcal{C}^{i}_{*,b}|\overset{n}{\cong}qq^{(q-1)q^{n-1}}, |\mathcal{C}%
^{i}_{*,b}|\overset{q}{\cong}qq^{(q-1)q^{n-1}};
\]

\[
|\mathcal{C}^{*}_{a,b}|\overset{n}{\cong}nq^{(q-1)q^{n-1}}, |\mathcal{C}%
^{*}_{a,b}|\overset{q}{\cong}nq^{(q-1)q^{n-1}};
\]

\[
|\mathcal{C}^{*}_{a,*}|\overset{n}{\cong}nqq^{(q-1)q^{n-1}}, |\mathcal{C}%
^{*}_{a,*}|\overset{q}{\cong}nqq^{(q-1)q^{n-1}};
\]

\[
|\mathcal{C}^{*}_{*,b}|\overset{n}{\cong}nqq^{(q-1)q^{n-1}}, |\mathcal{C}%
^{*}_{*,b}|\overset{q}{\cong}nqq^{(q-1)q^{n-1}};
\]

\[
|\mathcal{C}^{i}_{*,*}|\overset{n}{\cong}q^{2}q^{(q-1)q^{n-1}}, |\mathcal{C}%
^{i}_{*,*}|\overset{q}{\cong}q^{2}q^{(q-1)q^{n-1}};
\]

\[
|\mathcal{C}^{*}_{*,*}|\overset{n}{\cong}nq^{2}q^{(q-1)q^{n-1}},
|\mathcal{C}^{*}_{*,*}|\overset{q}{\cong}nq^{2}q^{(q-1)q^{n-1}};
\]

\end{thm}

\begin{proof}

The first two rows are previous lemma and corollary.

We will give a proof for the last row, the others are similar and easier.%

\[
|\mathcal{C}^{*}_{*,*}|=\sum^{q}_{k=1}(-1)^{k-1}U_{k}+\sum^{nq}_{k=1}%
(-1)^{k-1}V_{k}.
\]
\[
U_{1}=\frac{nq!}{(q-1)!}q^{(q-1)q^{n-1}}\sum_{t_{1}+t_{2}+...+t_{q}=1, 0\leq
t_{i}\leq q}\frac{1}{t_{1}!t_{2}!...t_{q}!}=nq^{2}q^{(q-1)q^{n-1}}.
\]
When $2\leq k\leq q$, we have
\[
U_{k}=\frac{nq!}{(q-k)!}q^{(q-k)q^{n-1}}\sum_{t_{1}+t_{2}+...+t_{q}=k, 0\leq
t_{i}\leq q}\frac{1}{t_{1}!t_{2}!...t_{q}!}\leq nq!q^{(q-2)q^{n-1}}\sum_{0\leq
t_{i}\leq q,i=1,2,...,q}1
\]
\[
=nq!q^{(q-2)q^{n-1}}(q+1)^{q}.
\]
So, $\lim_{n\to\infty}\frac{U_{k}}{U_{1}}=0$ for $2\leq k\leq q$.

$V_{1}=0$ since the condition of the sum is not satisfied.

When $2\leq k\leq nq$, we have
\[
V_{k}=q\sum_{k_{1}+...+k_{n}=k, 0\leq k_{i}\leq k-1,0\leq k_{i}\leq q}%
(\prod^{n}_{j=1}C(q,k_{j}))q^{\prod^{n}_{j=1}(q-k_{j})}\leq q\sum_{0\leq
k_{i}\leq q,i=1,2,...n}(nq!)q^{(q-1)^{2}q^{n-2}}%
\]
\[
=q(q+1)^{n}nq!q^{(q-1)^{2}q^{n-2}}.
\]
Hence,
\[
|\sum^{nq}_{k=1}(-1)^{k-1}V_{k}|\leq nqq(q+1)^{n}nq!q^{(q-1)^{2}q^{n-2}}.
\]

We obtain
\[
\lim_{n\to\infty}\frac{|\sum^{nq}_{k=1}(-1)^{k-1}V_{k}|}{U_{1}}\leq\lim
_{n\to\infty}\frac{nqq(q+1)^{n}nq!q^{(q-1)^{2}q^{n-2}}}{nq^{2}q^{(q-1)q^{n-1}%
}}=\lim_{n\to\infty}\frac{(q+1)^{n}nq!}{q^{(q-1)q^{n-2}}}=0.
\]
In summary, we obtain
\[
\lim_{n\to\infty}\frac{|\mathcal{C}^{*}_{*,*}|}{nq^{2}q^{(q-1)q^{n-1}}}=1
\]
In other words,%

\[
|\mathcal{C}^{*}_{*,*}|\overset{n}{\cong}nq^{2}q^{(q-1)q^{n-1}}.
\]

From the above proof, it is also clear that we have%

\[
\lim_{q\to\infty}\frac{|\mathcal{C}^{*}_{*,*}|}{nq^{2}q^{(q-1)q^{n-1}}}=1
\]

In other words,%

\[
|\mathcal{C}^{*}_{*,*}|\overset{q}{\cong}nq^{2}q^{(q-1)q^{n-1}}.
\]

\end{proof}

When $q=2$, the first part of the last row in the above theorem has been
obtained in \cite{Win}.

\section{Conclusion}

\label{sec4}

In this paper, we generalized the definition of Boolean canalyzing functions
to the functions over general finite fields $\mathbb{F}_{q}$. We obtain clear
characterization for all eight classes of canalyzing functions. Using
Inclusion and Exclusion Principle, we also obtain eight formulas for the
cardinality of these classes. The main idea is from \cite{Jar1} and
\cite{Win}. Actually, the characterization is motivated from a simple lemma in
\cite{Jar1}. The enumeration idea is a natural extension of \cite{Win}. By
specifying our results to the case $q=2,$we obtain the formula in \cite{Win},
and derive an interesting combinatorial identity. Finally, for a better
understanding to the magnitudes, we provide all the eight asymptotes of these
cardinalities as either $n\rightarrow\infty$ or $q\rightarrow\infty$.

$Acknowledgments$

This work was initiated when the first and the third authors visited Virginia
Bioinformatics Institute at Virginia Tech in June 2010. We thank Alan
Veliz-Cuba and Franziska Hinkelmann for many useful discussions. The first and
the third authors thank Professor Reinhard Laubenbacher for his hospitality
and for introducing them to Discrete Dynamical Systems. The third author was
supported in part by a Research Initiation Program (RIP) award at
Winston-Salem State University.

\end{document}